\documentclass{article}

\usepackage{amsmath,amssymb,amsfonts,amsthm}
\usepackage{graphicx,color}
\graphicspath{{./Figures/}}
\usepackage{psfrag}

\newcommand{\until}[1]{\{1,\dots, #1\}}
\newcommand{\subscr}[2]{#1_{\textup{#2}}}

\newcommand{\setdef}[2]{\{#1 \, : \; #2\}}
\newcommand{\map}[3]{#1: #2 \rightarrow #3}

\newcommand{\union}{\cup}
\newcommand{\intersection}{\cap}

\newcommand{\eps}{\varepsilon}

\newcommand{\real}{\mathbb{R}}
\newcommand{\realpositive}{\mathbb{R}_{>0}}
\newcommand{\realnonnegative}{\mathbb{R}_{\geq0}}

\newcommand{\integer}{\mathbb{Z}}

\newcommand{\Union}{\bigcup}
\newcommand{\Intersection}{\bigcap}

\newcommand{\bigsetdef}[2]{\left\{#1 \, : \; #2\right\}}
\newcommand{\ds}{\displaystyle}

{
\theoremstyle{definition}
\newtheorem{example}{Example}

}
\newtheorem{theorem}{Theorem}

\newtheorem{corollary}[theorem]{Corollary}
\newtheorem{proposition}[theorem]{Proposition}

\newcommand{\be}{\begin{equation}}
\newcommand{\ee}{\end{equation}}

\newcommand{\co}{\overline{\textup{co}}} 
\newcommand{\G}{\mathcal{G}} 
\newcommand{\E}{\mathcal{E}} 
\newcommand{\xave}{\subscr{x}{ave}} 
\newcommand{\Tcon}{\subscr{T}{con}} 
\newcommand{\Tempty}{\subscr{T}{empty}} 

\newcommand{\K}{\mathcal{K}}  

\newcommand{\qset}{\mathcal{S}} 

\newcommand{\aLow}{a_{\min}} 
\newcommand{\aUp}{a_{\max}} 

\newcommand{\borderI}{I^{\partial}} 
\newcommand{\intI}{\mathring{I}}
\newcommand{\plusI}{I^+}

\title{Continuous-time quantized consensus: convergence of Krasovskii solutions}
\author{Paolo Frasca\thanks{P. Frasca is with the Dipartimento di Matematica, Politecnico di Torino, Torino, Italy. \texttt{paolo.frasca@polito.it}. The work of the author was partly supported by MIUR under grant PRIN-20087W5P2K.
The author wishes to thank F. Ceragioli and two anonymous reviewers for their remarks.
}
}

\begin{document}
\maketitle

\begin{abstract} 
This note studies a network of agents having continuous-time dynamics with quantized interactions and time-varying directed topology. Due to the discontinuity of the dynamics, solutions of the resulting ODE system are intended in the sense of Krasovskii. 
A limit connectivity graph is defined, which encodes persistent interactions between nodes: if such graph has a globally reachable node, Krasovskii solutions reach consensus (up to the quantizer precision) after a finite time.
Under the additional assumption of a time-invariant topology, the convergence time is upper bounded by a quantity which depends on the network size and the quantizer precision. It is observed that the convergence time can be very large for solutions which stay on a discontinuity surface.
\end{abstract}


\section{Introduction}


Problems of consensus and coordination in networks have been widely studied during the last decade using a blend of tools from control theory and graph theory.
While linear consensus systems based on time-invariant networks are easy to understand, things become harder when the network topology depends on time, or when communication between nodes is affected by limited precision due to bandwidth constraints.
Consensus problems have been studied on time-dependent networks by a vast literature: we refer the reader to the early works~\cite{ROS-RMM:03c,LM:04}, as well as to the books~\cite{FB-JC-SM:09,MM-ME:10} for an introduction and to~\cite{JMH-JNT:11a} for recent related results. 
On the other hand, coordination and consensus have also been studied in systems subject to limited-precision effects, {\it i.e.}, to quantization. Most authors have focused on a variety of problems for discrete-time systems, including the analysis of convergence assuming static quantizers~\cite{AK-TB-RS:07}\cite{TCA-MJC-MGR:08}\cite{PF-RC-FF-SZ:08}\cite{RC-FF-PF-SZ:08}\cite{SK-JM:10}
and the design of effective dynamic quantization schemes~\cite{RC-FB-SZ:10}\cite{TL-MF-LX-JFZ:09}.
Quantized continuous-time systems, instead, have attracted attention more recently. Controllers based on quantizing the differences between the states of connected nodes are studied in~\cite{DVD-KHJ:10} under the assumption that the network topology is a tree, and in~\cite{GC-FLL-LX:11} using binary quantizers in a leader-following framework.
Quantized communication of states is instead considered in~\cite{FC-CDP-PF:10a} for static topologies and in~\cite{CDP:11} for dissipative systems. 

In the analysis of continuous-time quantized dynamics, the inherent discontinuity of the system right-hand side entails some mathematical difficulties, which are discussed in~\cite{JC:06b} and~\cite{FC-CDP-PF:10a}. The latter paper considers a simple continuous-time average consensus dynamics with time-invariant topology and uniform static quantizers, and  demonstrates that choosing a suitable definition of solution is essential to ensure that solutions are defined for all times and thus to permit a meaningful convergence analysis.
A natural and effective choice are Krasovskii solutions, which indeed are complete for every initial condition and converge to approximate consensus conditions under mild assumptions.

\subsection*{Statement of contributions}
After this literature review, we are able to present the contribution of this paper.
We study a coordination task for a network of agents having a scalar continuous-time dynamics, assuming that 
\renewcommand{\theenumi}{\roman{enumi}}
\renewcommand{\labelenumi}{(\theenumi)}
\begin{enumerate}
\item the interaction between the agents is weighted by time-dependent coefficients which represent a dynamical communication network; and 
\item connected agents can exchange information about their states only through a (static) quantizer.
\end{enumerate} Due to the quantization constraint, the goal of consensus between states can only be approximated up to the quantizer precision. 
Our main contribution consists in a sufficient condition for finite-time convergence of Krasovskii solutions to the best achievable approximation. This condition, presented in Theorem~\ref{th:general-conv}, is based on the connectivity of a suitable limit graph.
Compared with the referenced literature, our convergence result holds (i) under milder assumptions on the network connectivity; and (ii) for a larger class of quantizers.
Additionally, in Section~\ref{sec:uniform} the convergence result is specialized to uniform quantizers and to average-preserving dynamics. With the further assumption of time-invariant topology, we also derive an upper bound on the convergence time, which is inversely proportional to the quantizer precision and is exponentially increasing with the network size. The tightness of this bound is discussed in view of {\it ad hoc} examples and of the evidences in the literature.
We leave outside the scope of this paper the analysis of controllers based on quantization of differences, as well as the design of optimal controllers and quantizers, either dynamic or static.

\section{Mathematical tools: Graphs and ODEs}
In this section we provide some background in differential equations and graph theory.
For our analysis it is necessary to recall from~\cite{OH:79} a certain notion of solution to a --possibly discontinuous-- differential equation, which is based on defining a suitable differential inclusion.
Given\footnote{The symbols $\integer$, $\real$, $\realnonnegative$ $\realpositive$ denote the sets of integer, real, nonnegative and positive numbers, respectively. $\real^n$ denotes an $n$-dimensional Euclidean space. Writing $\real^A$, where $A$ is a set of cardinality $n$, we are indexing the components in the set $A$. Given $r\in \real$, the set of the (integer) multiples of $r$ is denoted by $r \integer$.} $\map{f}{\realpositive\times \real^N}{\real^N}$ and the differential equation $\dot x=f(t,x)$, we say that $\map{x}{J}{\real^N}$ solves this differential equation in the Krasovskii sense if $x(\cdot)$ is absolutely continuous and for almost every time $t$ in the interval $J\subset \realpositive$ satisfies the differential inclusion
$\dot x(t)\in \K f(t,x(t))$, 
where $$\K f(t,x)=\Intersection_{\delta>0} \co f(t,B(x,\delta)),$$
with $\co$ denoting the convex closure and $B(y,r)$ the Euclidean ball of radius $r$ centered in $y$. Here and elsewhere in the paper, ``almost every'' means ``except in a set of zero Lebesque measure''. A solution is said to be complete if $J=(0,+\infty).$
Note that we will also apply the Krasovskii operator $\K$ to autonomous functions $f(x)$. An example of the convexification induced by $\K$ is provided later in Figure~\ref{fig:Kq-map}.
A similar notion is that of Filippov solution~\cite[Definition~6]{AB-FC:99},
which is quite common in the literature but will not be used here. Indeed, every Filippov solution is also a Krasovskii solution, so that the results in this paper apply {\it a fortiori} to Filippov solutions.

Our analysis also involves graphs and weighted graphs. We introduce here the main notions which we shall use later: the reader is referred to the literature, for instance to~\cite{RC-FF-AS-SZ:08} or to the book~\cite{MM-ME:10}, for a more complete introduction. Given a finite set of vertices (or nodes)~$V$, a (directed) graph $G$ is a pair $(V,E)$ where $E\subset V\times V$ is the set of edges (or arcs). A weighted graph is triple $(V,E,A)$ which includes a weighted adjacency matrix $A\in\realnonnegative^{V\times V}$ with the consistency condition that $A_{uv}>0$ if and only if $(u,v)\in E.$ We also assume that $A_{uu}=0$ for all $u\in V.$
The Laplacian matrix associated to $A$ is a matrix $L\in\real^{V\times V}$ such that $L_{uv}=-A_{uv}$ if $u\neq v$ and $L_{uu}=\sum_{v\in V}A_{uv}$.
A sink is a node $u$ with no outgoing edge --that is, such that $E$ does not contain any edge of the form $(u,v).$
A path (of length $l$) from $u$ to $v$ in $G$ is an ordered list of edges $(e_1, \ldots, e_l)$ in the form $((u, w_1),(w_1,w_2), (w_2,w_3), \ldots, (w_{l-1},v)).$
If such a path exists, we say that $v$ can be reached from $u$.
A cycle is a path from a node to itself. A graph is said to be connected if for every pair of nodes $(u,v)$, either $v$ can be reached from $u$ or $u$ can be reached from $v$. Instead, a graph is said to be strongly connected if every two nodes can be reached from each other.
Given any directed graph $G=(V,E)$ we can consider its strongly connected components, namely maximal strongly connected subgraphs $G_k$, $k\in\until{s}$ with set of vertices $V_k \subset V$ and set of arcs $E_k = E \intersection (V_k \times V_k)$ such that the sets $V_k$ form a partition of $V$. These components may have connections among each other: in order to encode these connections we define a directed graph $\mathcal T(G)$ with set of vertices $\until{s}$ such that there is an arc from $h$ to $k$ if there is an arc in $G$ from a vertex in $V_h$ to a vertex in $V_k$. We observe that (i) $\mathcal T(G)$ has no cycle; (ii) $\mathcal T(G)$ is connected and has one sink if and only if there exists in $G$ a globally reachable node, {\it i.e.}, a node which can be reached from every other node.

\section{Problem statement and main result}
In this section we introduce the dynamics of interest, and we state and prove our main convergence result.
Let there be 
$N$ agents, indexed in a set $I$, and for any pair $(i,j)\in I\times I$, let $\map{a_{ij}(\cdot)}{\realnonnegative}{{0}\union [\aLow,\aUp]}$ be a measurable function, with $0<\aLow\le\aUp$.
These interaction functions naturally lead to the following definitions.
For every time $t$, we consider a weighted interaction graph $\G(t)=(I, \E(t), A(t))$, such that the $i,j$-th component of the matrix $A(t)$ is the value $a_{ij}(t),$ and $(i,j)\in \E(t)$ if and only if $a_{ij}(t)>0.$
Given the function $\G(t)$, we define --following~\cite{JMH-JNT:11a}-- an {\em unbounded interactions graph} $\G_\infty=(I, \E_\infty)$ 
by $$\E_\infty=\setdef{(i,j)\in I\times I}{\lim_{t\to+\infty}\int_{t_0}^{t}a_{ij}(s)ds=+\infty\quad \forall\, t_0\ge0}.$$
We observe that $\G_\infty$ is the graph whose edges connect the nodes which are connected in $\G(t)$ for an infinite duration of time.

For $i\in I$, we let $x_i(t)$ be a real variable and consider the dynamics
\be\label{eq:quantized-dynamics}
\dot x_i=\sum_{j\in I}{a_{ij}(t) (q(x_j)-q(x_i))}
\ee
where $\map{q}{\real}{\qset}$ is a {\em quantizer} mapping real numbers into a discrete\footnote{A subset $\qset\subset \real$ is said to be discrete if all its points are isolated. Examples include the set of the integers and every finite subset of $\real$. Note that if $\qset$ has no limit point in $\real$, then $\qset$ is discrete.} set.
System~\eqref{eq:quantized-dynamics} can also be rewritten in vector form as $$\dot x= -L(t) q(x),$$ where $x(t)\in \real^I$ is the state vector, $L(t)$ is the Laplacian matrix associated to the weighted adjacency matrix $A(t)$ and by a slight notational abuse, $q$ is defined to operate componentwise on vectors.
We consider for~\eqref{eq:quantized-dynamics} solutions in the sense of Krasovskii, which we have defined in the previous section, and thanks to the linearity of the Krasovskii operator $\K$, we have that a Krasovskii solution to~\eqref{eq:quantized-dynamics} is an absolutely continuous function of time which satisfies for almost every time the differential inclusion
$$\dot x\in - L(t) \K q(x).$$
By the current assumptions of boundedness on the functions $a_{ij}$, for any $\bar x\in \real^I$ there exists a complete Krasovskii solution $x(t)$ to~\eqref{eq:quantized-dynamics}, such that $x(0)=\bar x.$ Note, however, that there can be more than one of such solutions. In the rest of this paper, whenever we refer to a solution, we mean a complete solution.

\bigskip
After these preliminary observations, we are ready to state and prove that system~\eqref{eq:quantized-dynamics} reaches quantized consensus equilibria in finite time, provided the unbounded interactions graph has a globally reachable node.

\begin{theorem}[Finite-time quantized consensus]\label{th:general-conv}
Let $\qset$ be a subset of $\real$ with no limit point and $\map{q}{\real}{\qset}$ be a non-decreasing function. Let $x(t)$ be a Krasovskii solution to~\eqref{eq:quantized-dynamics}.
If $\mathcal T(\G_\infty)$ is connected and has only one sink, then there exist $\Tcon\ge0$ and $s^*\in \qset$ such that, for every $t\ge \Tcon$, $$s^*\in \K q(x_i(t))\quad  \text{for every $i\in I$}.$$
\end{theorem}
\begin{proof}
Without loss of generality, we may think of the elements of $\qset$ as indexed in a set $A$ of consecutive integers, in such a way that $\qset=\setdef{s_a}{a\in A}$ and $s_a<s_b$ if and only if $a<b$.
Let $\Delta_{\min}=\inf\setdef{|s_a-s_b|}{a,b\in A}.$ 
As there is no limit point of $\qset$, then $\Delta_{\min}>0.$ 

Given the solution $x(\cdot)$ and $z\in \real$, we define the following time-dependent subset of indices $$I_z(t)=\setdef{i\in I}{z\in \K q(x_i(t))},$$ and we let $m(t)=\min_{i\in I} \min \K q(x_i(t))$ and $M(t)=\max_{i\in I} \max \K q(x_i(t))$. By definition, $M(t)$ and $m(t)$ belong to $\qset$ and we denote $m(0)=s_m$ and $M(0)=s_M.$
The dynamics~\eqref{eq:quantized-dynamics} implies that, at almost every time $t$ and for all $i\in I$,
\be\label{eq:dotxi-in-proof}\dot x_i(t)\in\bigsetdef{\sum_{j\in I} a_{ij}(t) (z_j-z_i)}{z_k\in \K q(x_k(t))}.\ee
In particular, if $i\in I_{m(t)}(t)$, then $\dot x_i(t)\in [0,+\infty)$. Hence, $m(t)\ge m(0)$ for all $t\ge0$; similarly, we can deduce that $M(t)\le M(0).$ Our proof aims at showing that $m(t)$ actually increases until the system reaches an equilibrium: there exist $\Tcon\ge0$ and $s^*\in \qset$ such that $I_{s^*}(\Tcon)=I$. The same conclusion can be reached by an analogous argument based on $M(t)$.
Note that for all $t\ge0$ it holds $I=\Union_{s=s_m}^{s_M} I_s(t)$, but sets of the form $I_{s_h}(t)\intersection I_{s_{h+1}}(t)$ need not to be empty, in particular when some $x_k(t)$ is at a discontinuity point of $q$. 

In view of the last remark, we denote for brevity
$\borderI_{s_h}(t)=I_{s_h}(t)\intersection I_{s_{h+1}}(t)$ and $\intI_{s_h}(t)=I_{s_h}(t)\setminus I_{s_{h+1}}(t)$,
and we start our argument by considering the set $\intI_{s_m}(t)$ and claiming that
\begin{equation}\label{eq:invariance}
\intI_{s_{m}}(t_1)\supseteq  \intI_{s_{m}}(t_2)\qquad\text{for all $t_2\ge t_1\ge 0$.
}\end{equation} %
We show this fact by contradiction.
Let $x_0\in \real$ be the discontinuity point of $q$ such that $k\in \borderI_{s_m}(t)$ if and only if $x_k(t)=x_0.$ Assume, by contradiction, that there exists an agent $i\in I$ such that $x_i(t_1)>x_0$ and $x_i(t_2)<x_0$. Then there are three consequences:
(i) by continuity, there exists $t'\in (t_1,t_2)$ such that $x_i(t')=x_0$;
(ii) consequently, $x_i(t_2)=x_0+\int_{t'}^{t_2} \dot x_i(s)ds$;  
(iii) $x_i(t)<x_0$ for $t\in (t',t_2)$, and since $i\in \intI_{s_{m}}(t)$, then necessarily $\dot x_i(t)\ge0$.
But (ii) implies that, for a set of times of positive measure, $\dot x_i(t)<0,$ which is a contradiction. 
We conclude that~\eqref{eq:invariance} holds and that if an agent $k$ reaches $I_{s_{m}}(t)$ (from the right), she necessarily has to stop at the border of the corresponding interval. 

Next, we want to prove that there exist times when the inclusion~\eqref{eq:invariance} is strict.
We define the set of the agents whose state is ``strictly larger'' than $s_m$ as $$\plusI_{s_{m}}(t)=\displaystyle\left(\Union_{h=m+1}^{M} I_{s_h}(t) \right)\setminus I_{s_{m}}(t)$$
and 
$\Tempty=\inf\setdef{t\ge0}{\plusI_{s_{m}}(t)=\emptyset}.$ If $\Tempty$ is finite, then $\plusI_{s_{m}}(t)=\emptyset$ for every $t\ge \Tempty$. Then $I_{s_m}(\Tempty)=I$ and we conclude that $\Tcon=\Tempty$ and $s^*=s_{m}$, completing the proof. 
Otherwise, we proceed with our argument and assume\footnote{If $\intI_{s_m}(0)=\emptyset$, then there is nothing to prove: since in this case $I_{s_m}(0)\subseteq \intI_{s_{m+1}}(0)$, we can start our argument from $\intI_{s_{m+1}}(0).$} by contradiction that 
$\intI_{s_{m}}(t)=  \intI_{s_{m}}(0)$ for all $t>0$.
We also temporarily assume that $\G_\infty$ is strongly connected: the argument will be extended at the end of the proof. Then, thanks to the strong connectivity of $\G_\infty$, we can find an arc $(i,j)\in \E_\infty$ such that $i$ belongs to $\intI_{s_{m}}(0)$ and $j$ does not. 
As a consequence of~\eqref{eq:invariance}, $j\notin \intI_{s_{m}}(t)$ for all $t\ge0,$
and by contradiction we know that $i\in \intI_{s_m}(t)$ for all $t\ge0.$
Notice that for almost every $t\ge0$, 
$$\dot x_i(t)\ge a_{ij}(t)\, \big(v_j(t)-q(x_i(t))\big),$$
where $v_j(t)\in \K q(x_j(t))$ is the realization of the inclusion in~\eqref{eq:dotxi-in-proof}.
Define
$ J_\partial=\setdef{t\ge0}{j\in  \borderI_{s_{m}}(t)} $
and 
$ J_+=\setdef{t\ge0}{j\in  \plusI_{s_{m}}(t)}.$
Then 
\begin{align}x_i(t) \ge\, & x_i(0)+\int_{0}^{t} a_{ij}(s) \big(v_j(s)-q(x_i(t))\big)ds \label{eq:has-to-diverge-a}\\
\nonumber
=\, & x_i(0)+\int_{J_\partial \intersection (0,t)} a_{ij}(s) \big(v_j(s)-q(x_i(t))\big) ds +\int_{J_+ \intersection (0,t)} a_{ij}(s) \big(v_j(s)-q(x_i(t))\big) ds\\
\nonumber\ge\, & x_i(0)+\Delta_{\min} \int_{J_\partial \intersection (0,t)} a_{ij}(s) \alpha_j(s)ds +\Delta_{\min}\int_{J_+ \intersection (0,t)} a_{ij}(s) ds 
\end{align}
where in the last inequality we have used the fact that if $s\in J_\partial$, then $$v_j(s)=s_m (1-\alpha_j(s)) + s_{m+1} \alpha_j(s)=s_{m} + \alpha_j(s)(s_{m+1}-s_{m}),$$ and $\alpha_j(s)$ is a measurable function taking values in $[0,1].$
Let $q^{-1}(s_m)$ denote the pre-image of $s_m$ under $q$. If $\sup{q^{-1}(s_m)}=+\infty$, then necessarily $s_m=\max \qset$ and the proof is completed since $I_{s_m}(0)=I$. Otherwise, 
we aim to show that the right-hand side of~\eqref{eq:has-to-diverge-a} is divergent as $t\to\infty$.
If $J_+$ has infinite measure, divergence is clear from the assumption $a_{ij}(s)\ge \aLow.$ Otherwise, $\lim_{t\to\infty} \int_{J_+ \intersection (0,t)} a_{ij}(s) ds<\infty$ and instead $J_\partial$ has infinite measure: we want to use this fact, together with a lower bound on $\alpha_j(s)$. To obtain such an estimate, we note that Equation~\eqref{eq:invariance} implies that $\dot x_k(t)=0$ for almost every $t$ such that $k\in \borderI_{s_{m}}(t)$. 
Then, for almost every $s\in J_\partial$ it holds 
$\dot x_j(s)=0$, and the equality
$$\dot x_j(s)= \sum_{k} a_{jk} (v_k(s)-v_j(s))=
\sum_{k} a_{jk} \big(v_k(s)-\alpha_j(s) (s_{m+1}-s_m) - s_m\big)
$$
implies that 
\begin{align*}
\alpha_j(s)=& \frac{ \sum_{k}a_{jk}(s) (v_k(s)-s_m)}  {(s_{m+1}-s_m) \sum_{k}a_{jk}}\\
\ge& \frac
{\sum_{k\in \borderI_{s_{m}}(s)} a_{jk}(s) \alpha_k(s) + \sum_{k\in \plusI_{s_m}(s)} a_{jk}(s)}
 { \sum_{k}a_{jk} }
\\ \ge &
 \frac{\aLow}{N \aUp} \beta_k(s),
\end{align*}
where $\beta_k(s)$ is defined as follows. Let $l\in I$ be such that $(k,l)\in \E_{\infty}$ and for all $t$ in a set of times of infinite measure either $l\in \plusI_{s_m}(t)$ or $l\in \borderI_{s_m}(t)$. In the former case $\beta_k(s)=1$, in the latter $\beta_k(s)=\alpha_l(s).$
By the connectivity assumption, there exists an infinite-measure set of times $J$ such that for $s\in J$ there is a path in $\G(s)$ from $j$ to a node in $\plusI_{s_m}(s)$, and by a recursive reasoning along this path, we conclude that for $s\in J$ it holds $\alpha_j(s)\ge \left( \frac{ \aLow}{N \aUp} \right)^N$.
From~\eqref{eq:has-to-diverge-a} and the last inequality we can deduce
\begin{align}\label{eq:has-to-diverge-final}
x_i(t)
\ge &\, x_i(0)+\Delta_{\min} \int_{J\intersection (0,t)} \aLow \left( \frac{\aLow}{N \aUp} \right)^N.
\end{align}
This inequality implies that $x_i(t)$ diverges as $t\to+\infty$, which contradicts the fact that $x_i(t)\le M(0)$ for all $t\ge0.$
We conclude that there exists $T'>0$ such that for all $t\ge T'$ it holds $i\not\in \intI_{s_m}(t)$ and
 $\intI_{s_m}(t)\subsetneq\intI_{s_m}(0).$
Repeating this argument for every element of $\intI_{s_m}(0)$, we obtain that there exists $T_0>0$ such that $\intI_{s_{m}}(T_0)=\emptyset$.
%
%
Afterwards, the same reasoning which has been applied to $\intI_{s_m}$ can be applied, with straightforward modifications, to $\intI_{s_{m+1}}$, $\intI_{s_{m+2}}$, \ldots, showing that there exists a sequence of times $T_k$ such that $\intI_{s_{m+k}}(T_k)=\emptyset.$ Since $M(t)\le s_{M}$, then the sequence of $T_k$'s must be finite.
This implies that there exist $\Tcon$ and $s^*$ such that $I_{s^*}(\Tcon)=I$, under the assumption of strong connectivity of $\G_\infty$.

In order to complete the proof, we still have to relax the connectivity condition.
If $\G_\infty$ is not strongly connected, the above argument may fail, because at some time it may be impossible to find an arc $(i,j)$ coming out of the set of minima --say, the set $\intI_{s_m}(0)$. But in such a case, necessarily the sink component is a subset of $\intI_{s_m}(0)$. Then, since it is assumed that there is only {\em one} sink, it is still possible to conclude by applying the analogous argument based on the maximal value $M(t)$.
%
%
\end{proof}

Note that the assumptions of Theorem~\ref{th:general-conv} about $\qset$ are satisfied, for instance, when $\qset$ is a finite set or when $\qset=\Delta \integer$. The latter important special case is the topic of the next section.


\section{Uniform quantizers}\label{sec:uniform}
In this section, we assume that the states are communicated via a uniform quantizer, and we derive from Theorem~\ref{th:general-conv} a more precise convergence result. After that, we study the case of average-preserving dynamics, and we estimate the convergence time $\Tcon.$

Let then $q$ be the uniform quantizer with precision $\Delta>0$, that is the map $\map{q}{\real}{\Delta\integer}$ such that 
\be\label{eq:def-unif-quant}q(z)=\left\lfloor \frac{z}{\Delta}+\frac{1}{2} \right\rfloor\Delta.\ee
The maps $q$ and $\K q(x)$ are illustrated in Figure~\ref{fig:Kq-map}.

\begin{figure*}[htb]
\begin{center} \begin{tabular}{cc}
{\psfrag{xlabel}{\footnotesize $x$} 
\psfrag{ylabel} [][][1][-90]{\hspace{-5mm} \footnotesize $q(x)$}
    \includegraphics[width=.49\textwidth]{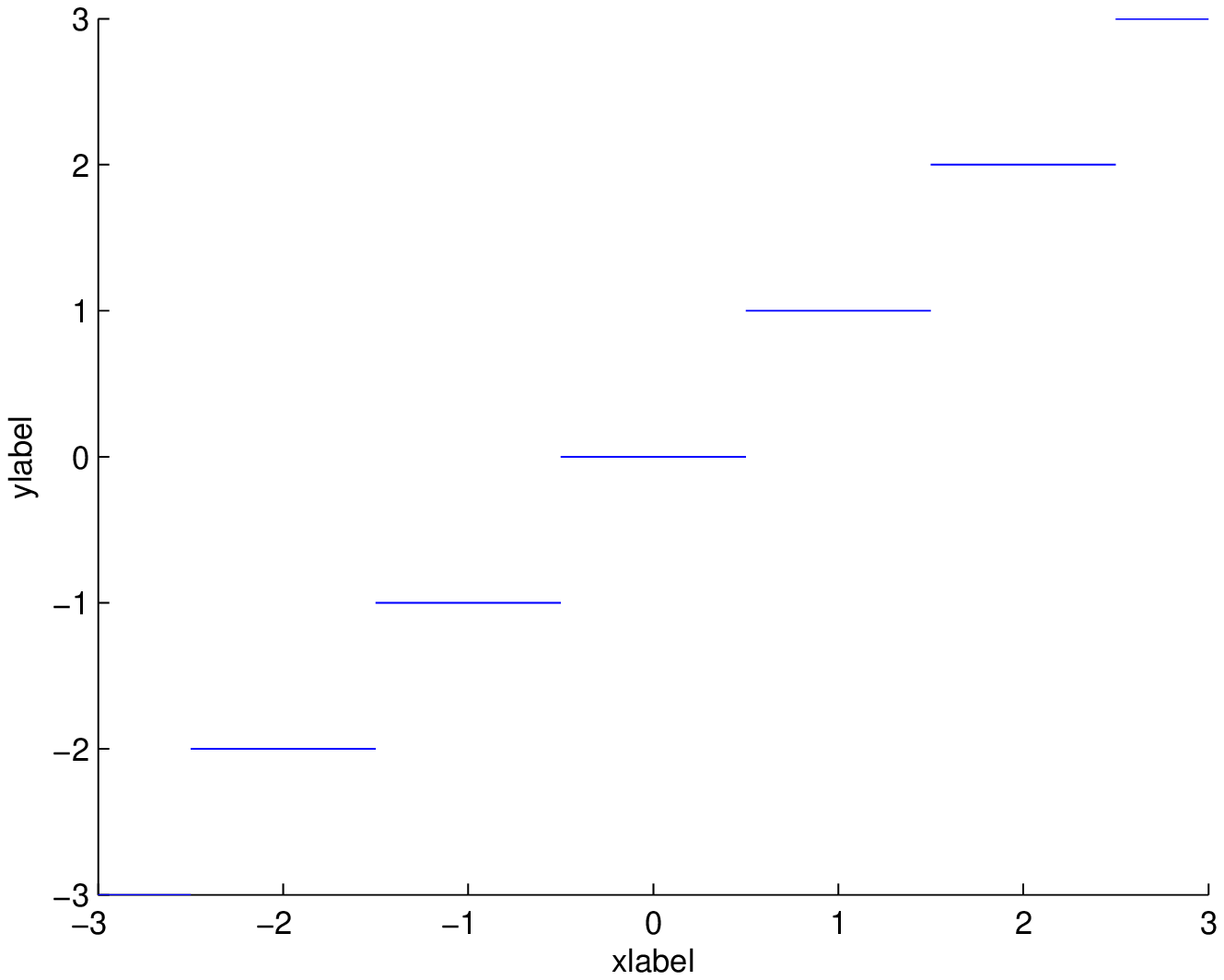}}&
{    \psfrag{xlabel}{\footnotesize $x$} 
\psfrag{ylabel}[][][1][-90]{\hspace{-5mm} \footnotesize$\K q(x)$}
    \includegraphics[width=.49\textwidth]{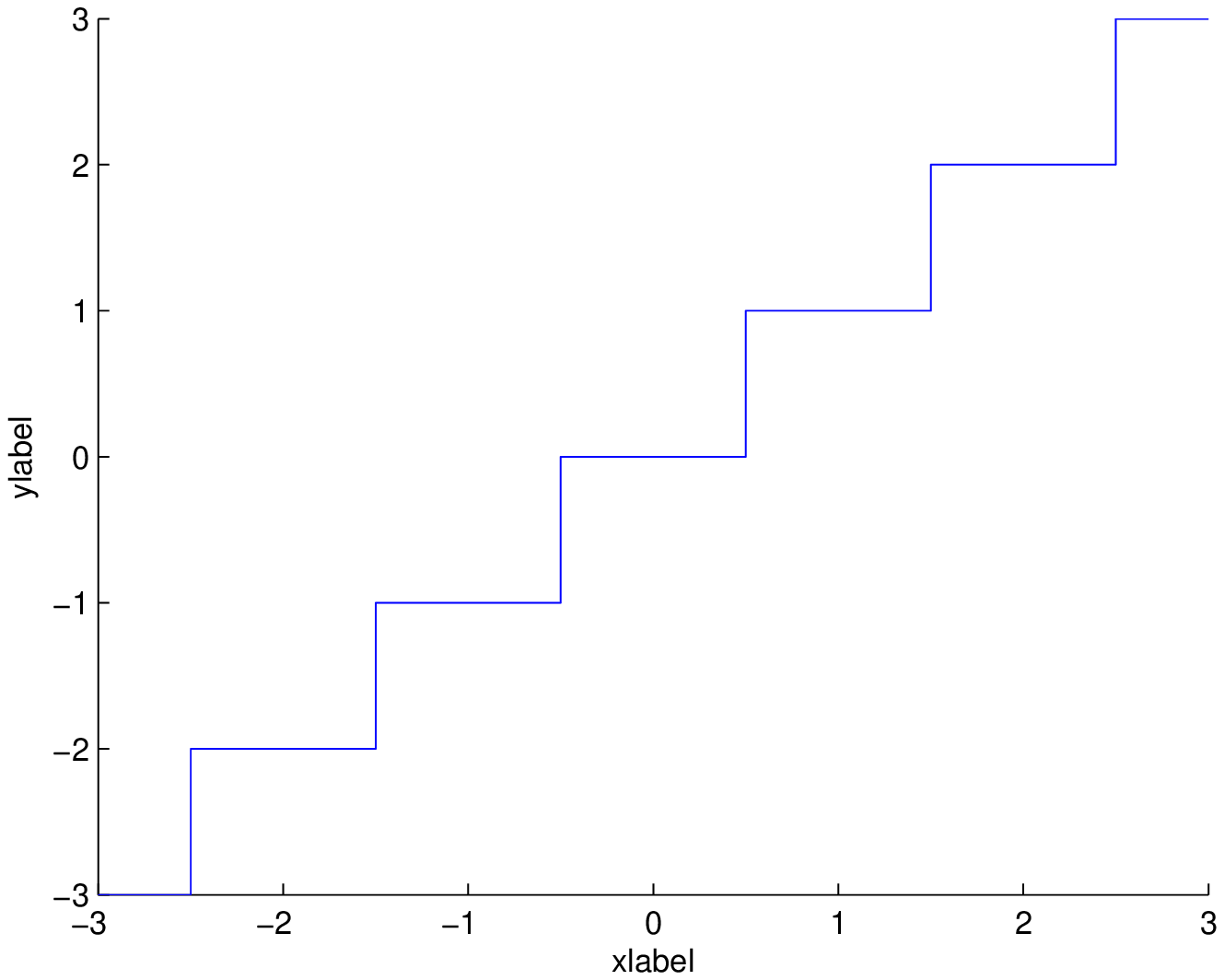}     }
    \end{tabular}
    \caption{Visualization of the map $q(x)$ in~\eqref{eq:def-unif-quant} and the corresponding set-valued map $\K q(x)$, when $\Delta=1$.}\label{fig:Kq-map}
\end{center} \end{figure*}

\begin{corollary}[Uniform quantizers]\label{corol:unif-conv}
Let $x(t)$ be a Krasovskii solution to~\eqref{eq:quantized-dynamics} and $q$ be defined as in~\eqref{eq:def-unif-quant}. If $\mathcal T(\G_\infty)$ is connected and has only one sink, then there exist $\Tcon\ge0$ and $q_\infty\in \Delta \integer$  such that for all $t\ge \Tcon,$
$$ x_i(t)\in \left[q_\infty - \frac\Delta2 , q_\infty + \frac\Delta2 \right]\qquad \text{for all } i\in I.$$
\end{corollary}

\begin{proof}
Since $\qset=\Delta\integer$, Theorem~\ref{th:general-conv} implies that 
there exist a nonnegative time $\Tcon$ and an integer $k$ such that for all $t\ge \Tcon,$ it holds
$$ k\Delta\in \K q(x_i(t)) \qquad \text{for all } i\in I.$$
This fact is equivalent to the statement of the corollary.
%
%
\end{proof}

\subsection{Average consensus}
In many applications one is concerned, rather than with mere convergence, with convergence to a certain target value, which is a function of the initial condition. For instance, the target can be the average of the initial states: this problem is referred to as the {\em average consensus problem}, and is studied in the next result.

\begin{corollary}[Average-preserving dynamics]\label{corol:average}
Let $x(t)$ be Krasovskii solution to~\eqref{eq:quantized-dynamics} and $q$ as in~\eqref{eq:def-unif-quant}. Define $\xave(t)=\frac 1N\sum_{j\in I}x_j(t)$. 
If $\mathcal T(\G_\infty)$ is connected and has only one sink, and $$\sum_{j\in I}{ a_{ij}(t)}=\sum_{i\in I}{ a_{ij}(t)} \qquad\text{for almost every $t\ge 0$},$$ then $\xave(t)=\xave(0)$ for every $t>0$ and
the conclusion of Corollary~\ref{corol:unif-conv} holds. Moreover, if $\xave(0)\neq (k+\frac12) \Delta$ for every $k\in \integer$, then $q_\infty=q(\xave(0))$, whereas if 
$\xave(0)=(h+\frac12)\Delta$ for some $h\in \integer$, then $x_i(\Tcon)=\xave(0)$ for every $i\in I.$
\end{corollary}
\begin{proof}
By linearity, for almost every $t>0$
\begin{align*}
\frac{d}{dt} \xave(t)\in \K\left(\frac1N \sum_{i\in I} \sum_{j\in I} (a_{ij}(t)-a_{ji}(t))q(x_j(t))\right).
\end{align*}
By the assumption on the $a_{ij}$'s, this implies that $\frac{d}{dt} \xave(t)=0$ for almost every $t>0$, so that the average is preserved.
Corollary~\ref{corol:unif-conv} then implies that $\left[q_\infty - \frac\Delta2 , q_\infty + \frac\Delta2 \right] \ni \xave(\Tcon)=\xave(0)$. If in particular $\xave(0)\in \left(q_\infty - \frac\Delta2 , q_\infty + \frac\Delta2 \right),$ then it is clear that $q(\xave(0))=q_\infty$. Otherwise, being $\xave(\Tcon)$ at the border of the interval, necessarily all $x_i(\Tcon)$ must coincide.
\end{proof}
Note that Corollary~\ref{corol:average} provides a formula for the limit (quantized) value, and also a sufficient condition to achieve exact consensus between the states.
Corollary~\ref{corol:average} improves on earlier convergence results available in the literature about average consensus of Krasovskii solutions (cf.~\cite[Proposition~4]{FC-CDP-PF:10a}), as it shows finite-time convergence for every initial condition and allows for time-dependent topologies.

\subsection{Convergence time}
In order to estimate the convergence time in Corollary~\ref{corol:unif-conv}, we restrict ourselves to consider {\em time-invariant} topologies, in the following sense.  We assume that for every pair $(i,j)$, either $a_{ij}(t)=0$ for all $t\ge0$ or $a_{ij}(t)\in[\aLow,\aUp]$ for all $t\ge0$, so that we may write $\G(t)=(I, \E, A(t))$ and $\G_\infty=(I,\E)$.

\begin{proposition}[Estimate of $\Tcon$]\label{prop:Tcon-est}
 Let $x(t)$ be Krasovskii solution to~\eqref{eq:quantized-dynamics} and $q$ as in~\eqref{eq:def-unif-quant}. 
Assume that $\G(t)$ has time-invariant topology, $\mathcal T(\G_\infty)$ is connected and has only one sink. Then,
\be\label{eq:Tcon}
\Tcon\le  \frac{1}{\Delta}\frac{N}{\aLow}  \left(\frac{N \aUp}{\aLow}\right)^N\, \max_{i,j\in I}|q(x_i(0))-q(x_j(0))|.
\ee 
\end{proposition}
\begin{proof}
The proof is based on specializing the proof of Theorem~\ref{th:general-conv} to the case at hand: we refer to that proof using the same notation.
Equation~\eqref{eq:has-to-diverge-final} becomes, being the graph topology time-invariant,
$$
x_i(t)
\ge x_i(0)+\Delta \int_{J\intersection (0,t)} \aLow \left( \frac{\aLow}{N \aUp} \right)^N ds
\ge m(0)-\frac12 \Delta + \Delta \aLow \left( \frac{\aLow}{N \aUp} \right)^N t. 
$$
Then, considering the sequence of $T_k$'s, we argue that $T_k-T_{k-1}\le \frac{N}{\aLow} \left(\frac{N \aUp}{\aLow}\right)^N$ for every $k\ge1$, as every quantization interval contains at most $N$ agents. On the other hand, $k$ needs not to be larger than $(M(0)-m(0))/\Delta$. These remarks prove the statement.
\end{proof}


Next, we want to discuss the tightness of estimate~\eqref{eq:Tcon}, in terms of the dependence on $N$ and on $\Delta$.
The parameter $\Delta$ represents the quantizer precision and, in view of Corollary~\ref{corol:unif-conv}, also the accuracy which is achievable in approximating the consensus. The bound~\eqref{eq:Tcon} allows for a convergence time which is polynomial in $\Delta$: the following example shows that there exist families of solutions which meet the bound, exhibiting a convergence time proportional to $\Delta^{-1}.$
Indeed, for every $N$ we can find a weighted graph $\G$ and an initial condition $\bar x$ such that for a certain solution such that $x(0)=\bar x$,
$$\Tcon\ge\frac18 \frac{N}{\aLow\,\Delta} \max_{i,j\in I}|q(x_i(0))-q(x_j(0))| .$$
%
%
\begin{example}[Slow convergence: $\Tcon\sim \Delta^{-1}$]\label{ex:Tcon-Delta}
We let $N\ge 3$, $I=\until{N}$ and we assume the topology to be a line graph, namely 
$$a_{ij}=\begin{cases}
1 &\text{if}\:\: i=1 \:\text{and}\: j=2\\
1 &\text{if}\:\: 2\le i\le N-1 \:\text{and}\: j=i-1, i+1\\
1 &\text{if}\:\: i=N \:\text{and}\: j=N-1\\
0 &\text{otherwise}.
\end{cases}$$
Note that the resulting dynamics~\eqref{eq:quantized-dynamics} preserves the average of the states. Regarding the initial condition, we assume $x_i(0)= \Delta (i-1)$ for all $i\in I$.
In the analysis of the resulting system, we think of the agents as arranged on a line and we only describe the evolution of the leftmost agents (1,2,\ldots,$\lfloor N/2\rfloor$), the evolution of the others being symmetrical. 
For early positive times, all agents are still except agent $1$ which moves to the right with constant speed $\Delta$. Then, at time $T'=\frac12$ we have that $x_1(T')=\Delta/2$, that is agent $1$ reaches the border of the first quantization interval. Since $\K q(x_1(T'))\ni \Delta$, there is one Krasovskii solution such that for $t\in(T',2T')$, $x_1(t)$ is constant while agent 2 moves to the right until it reaches $x_2(2T')=3\Delta/2$, so that $\K q(x_2(2T'))\ni 2\Delta$ and $\K q(x_1(2T'))\ni \Delta$. Then, for $t\in (2 T',4 T')$ the only agent on the move is again agent $1$, until $x_1(4 T')=3\Delta/2$. At time $t=4\,T'$, the two agents have the same state $x_2(4T')=x_1(4T').$ 
After this time, agents 3, 2 and 1 move to the right during three successive time  intervals, so that at $t=9 T'$ they are all collocated as $x_1(t)=x_2(t)=x_3(t)=5\Delta/2.$ By repeating this reasoning, we observe that the constructed solution $x(\cdot)$ reaches the limit configuration of Corollary~\ref{corol:average} at time 
$$\Tcon=\frac12\sum_{k=0}^{\lfloor \frac{N}2\rfloor} (1+2 k)
=\frac12\left\lfloor \frac{N}2\right\rfloor \left(\left\lfloor \frac{N}2\right\rfloor+2\right)
\ge\frac18 N(N-1).$$
Since $q(x_N(0))-q(x_1(0))=(N-1)\Delta,$ then 
$\ds\Tcon\ge\frac18 N \frac{q(x_N(0))-q(x_1(0))}{\Delta}.$
\qed
\end{example}

On the other hand, $N$ is the number of agents, and the bound~\eqref{eq:Tcon} allows for a convergence time which is exponential in $N$. The following example provides a family of solutions such that 
\be\label{eq:Tcon-expN}\Tcon\ge C\, 2^N,\ee
for a positive constant $C$. 
We observe that in order to have an exponential-in-$N$ convergence time, the solution must stay on a discontinuity of the right-hand side for a finite duration of time.
\begin{example}[Slow convergence: $\Tcon\sim e^{N}$]\label{ex:Tcon-N}
We let $I=\until{N}$ and we assume that, given $0<a\le b$ 
$$
\begin{cases}
\dot x_1= a \big(q(x_{2})-q(x_{1})\big) \\
\dot x_i= a \big(q(x_{i+1})-q(x_{i})\big) + b \big(q(x_1)-q(x_i)\big) &\text{if}\:\: 2\le i\le N-1\\
\dot x_N=0.
\end{cases}
$$
We also assume that the quantizer is uniform with $\Delta=1$ and that the initial condition is 
$$\begin{cases}
x_1(0)=0\\
x_i(0)= \frac12 &\text{if}\:\: 2\le i\le N-1\\
x_N=1.
\end{cases}$$
Note that $x_i(0)$ is on a discontinuity point of $q$ for $2\le i\le N-1$: then the Krasovskii convexification is nontrivial and we have $\dot x= -L z$, denoting the convexified values as  $z_i=(1-\alpha_i)\times 0 + \alpha_i \times 1= \alpha_i$.
One can immediately verify that there exists a Krasovskii solution $x(\cdot)$ having the following properties:
\renewcommand{\theenumi}{\alph{enumi}}
\begin{enumerate}
\item for every $t\ge0$, it holds that $x_N(t)=1$ and $x_i(t)=\frac12$ if $2\le i\le N-1$;
\item $\alpha_i=\left(\frac{a}{a+b}\right)^{N-i}$ for all $2\le i\le N-1$ and for $t\le \Tcon$;
\item $\dot x_1(t)= a \left(\frac{a}{a+b}\right)^{N-1}$ almost always for $t\le \Tcon$;
\item  at time $\Tcon=\frac1{2a} \left(\frac{a+b}{a}\right)^{N-1}\ge \frac1{2a} 2^{N-1} $ the agents reach quantized consensus in the interval $[1/2,1]$.
\end{enumerate}
Then~\eqref{eq:Tcon-expN} follows choosing $C=\frac1{4a}.$ \qed

\end{example}

\bigskip
The qualitative behavior of the convergence time of Krasovskii solutions, outlined above, should be contrasted with that of nonquantized consensus dynamics.
Let $\subscr{T}{con}^\eps$ be the time for convergence within a precision $\eps$ in a suitable norm. Then, consensus dynamics without quantization typically yield a logarithmic dependence on $\eps$, 
 $$ \subscr{T}{con}^\eps \le C\,\log{\eps^{-1}},$$
where $C$ is a constant which depends on the initial condition and on the topology of the interaction graph, and entails a dependence on $N$ which is at most polynomial.

We conclude that our theoretical results predict a qualitative degradation of convergence speed due to quantization. 
However, Proposition~\ref{prop:Tcon-est} is intrinsically a worst-case result, and not every solution needs to achieve the performance bound.
Indeed, it is argued in~\cite[Remark~5]{FC-CDP-PF:10a} that, far from the equilibria, the quantized dynamics converges exponentially fast and has the same rate of convergence as the nonquantized linear consensus dynamics. This is confirmed by simulations reported in the same paper, which show logarithmic convergence times in both cases.
These remarks entail no contradiction: far away from the equilibria the quantized dynamics is well approximated by the nominal linear dynamics, and the effect of quantization can be studied as a bounded disturbance (cf.~\cite{PF-RC-FF-SZ:08,DB-LG-RP:09,AG-AG:11}). On the other hand, in a neighborhood of the equilibria the approximation is no longer good and the consequences of quantization may fully come out, as we have shown above.

\section{Summary and future work}
This paper has demonstrated that a mathematical framework combining graph theory and Krasovskii differential inclusions can be useful to solve problems of distributed control with quantized communication. Complete Krasovskii solutions of quantized consensus dynamics  exist for any initial condition, and it is possible to study their converge to equilibria of ``practical consensus''. Under a mild connectivity assumption, which translates to the unbounded interactions graph the usual connectivity condition for consensus on static networks, solutions are shown to reach a neighborhood of consensus after a finite time. The size of such neighborhood only depends on the quantizer, and can thus be made arbitrarily small by design. On the other hand, the convergence time can be exponentially increasing in the number of nodes for some solutions which slide on a surface of discontinuity of the dynamics.

A few natural generalizations of the present work would be of interest: we briefly mention three of them.
\begin{enumerate}
\item
In this paper, the states of the agents are communicated through a non-smooth map which is a quantizer, that is, whose range is a discrete space. However, our proof technique based on monotonicity properties seems to be promising for studying convergence of systems featuring more general non-smooth interaction maps. 
\item 
Theorem~\ref{th:general-conv} states sufficient conditions for consensus: is it then natural to ask whether these assumptions are necessary. 
While it is clear that the connectedness of $\G_\infty$ is necessary for consensus, we believe that the argument of Theorem~\ref{th:general-conv} can be extended in such a way to relax the non-degeneracy assumption $\aLow>0$. 
A sufficient connectivity condition would then be: there exist $T>0$, $\delta>0$ and a graph $G=(I,E)$ which has a globally reachable node and is such that if $(i,j)\in E$, then  $\int_{t_0}^{t_0+T}a_{ij}(t)dt > \delta$ for every $t_0>0$. 
We leave the proof of this extension to future research.
On the other hand, when $\G_\infty$ is not connected but is {\em cut-balanced} in the sense of~\cite[Assumption~1]{JMH-JNT:11a}, we expect results of partial consensus and clusterization~\cite{VDB-JMH-JNT:09a,FC-PF:10}.

\item In this work, connectivity is a function of time determined by an exogenous signal. However, there are applications in which connectivity between agents is state-dependent. Which would be the convergence properties of quantized continuous-time dynamics on a state-dependent network described by interaction functions of type $a_{ij}(t,x)$? This investigation may have broad applications, including rendezvous and coordination problems in robotic networks where the ability to communicate depends on the robot locations~\cite{FB-JC-SM:09,CDP-MC-FC:11}, and modeling opinion dynamics with limited verbalization capabilities~\cite{DU:03} in social networks.
\end{enumerate}


\bibliographystyle{plain}

\end{document}